\documentclass[sn-mathphys-num]{sn-jnl}

\usepackage{graphicx}%
\usepackage{multirow}%
\usepackage{amsmath,amssymb,amsfonts,bbm}%
\usepackage{amsthm}%
\usepackage{mathrsfs}%
\usepackage[title]{appendix}%
\usepackage{xcolor}%
\usepackage{textcomp}%
\usepackage{manyfoot}%
\usepackage{booktabs}%
\usepackage{algorithm}%
\usepackage{algorithmicx}%
\usepackage{algpseudocode}%
\usepackage{listings}%

\theoremstyle{thmstyleone}%
\newtheorem{theorem}{Theorem}
\newtheorem{proposition}[theorem]{Proposition}%
\newtheorem{corollary}[theorem]{Corollary}

\theoremstyle{thmstyletwo}%
\newtheorem{example}{Example}%
\newtheorem{remark}{Remark}%

\theoremstyle{thmstylethree}%

\begin{document}
\newcommand {\eps} {\varepsilon}
\newcommand {\Z} {\mathbbm{Z}}
\newcommand {\R} {\mathbbm{R}}
\newcommand {\T} {\mathbbm{T}}
\newcommand {\Q} {\mathbbm{Q}}
\newcommand {\N} {\mathbbm{N}}
\newcommand {\C} {\mathbbm{C}}
\newcommand {\K} {\mathbbm{K}}
\newcommand {\I} {\mathbbm{I}}
\newcommand {\dist} {{\rm{dist}}}
\newcommand {\cl}{\mathrm{cl}}
\newcommand {\PP} {\mathbbm{P}}
\newcommand {\ang} {\measuredangle}
\newcommand {\e} {{\rm{e}}}
\newcommand {\rank} {{\rm{rank}}}
\newcommand {\Span} {{\mathrm{span}}}
\newcommand {\card} {{\rm{card}}}
\newcommand {\ED} {\mathrm{ED}}
\newcommand {\cA} {\mathcal{A}}
\newcommand {\cO} {\mathcal{O}}
\newcommand {\cF} {\mathcal{F}}
\newcommand {\cC} {\mathcal{C}}
\newcommand {\cN} {\mathcal{N}}
\newcommand {\cV} {\mathcal{V}}
\newcommand {\cG} {\mathcal{G}}
\newcommand {\cB} {\mathcal{B}}
\newcommand {\cD} {\mathcal{D}}
\newcommand {\cP} {\mathcal{P}}
\newcommand {\cQ} {\mathcal{Q}}
\newcommand {\cW} {\mathcal{W}}
\newcommand {\cT} {\mathcal{T}}
\newcommand {\cI} {\mathcal{I}}
\newcommand {\Sn}[1] {\mathcal{S}^{#1}}
\newcommand {\range} {\mathcal{R}}
\newcommand {\kernel} {\mathcal{N}}
\newcommand{\one}{\mathbb{1}}
\renewcommand{\thefootnote}{\fnsymbol{footnote}}
\newcommand{\rle}{\rotatebox[origin=c]{-90}{$\le$}}
\newcommand{\rl}{\rotatebox[origin=c]{-90}{$<$}}
\newcommand{\rg}{\rotatebox[origin=c]{-90}{$=$}}

\title[On  generalized metrics of Vandermonde type]{On  generalized metrics of Vandermonde type}


\author*[1]{\fnm{Wolf-J\"urgen} \sur{Beyn}}\email{beyn@math.uni-bielefeld.de}



\affil*[1]{\orgdiv{Department of Mathematics}, \orgname{Bielefeld University}, \orgaddress{\street{POBox 100131}, \city{Bielefeld}, \postcode{D-33615},  \country{Germany}}}




\abstract{ In a series of papers \cite{Gaehler63,Gaehler64,Gaehler69},  S. G\"ahler defined  and investigated so-called $m$-metric spaces and
  their topological properties. An $m$-metric assigns to any  tuple of $m+1$ elements a real value (more generally
  an element in a partially odered set)  which
  satisfies the generalized metric axioms of semidefiniteness, symmetry, and simplex inequality. In this contribution we consider a
  new type of generalized metric which is based on the Vandermonde determinant.
  We present some remarkable  geometric consequences of the corresponding simplex inequality in the complex plane.
  Then we show that the Vandermonde principle of construction extends to linear spaces of arbitrary dimension by using symmetric multilinear maps  of degree $\frac{1}{2}m(m+1)$.  In particular, we analyze when this
  generalized metrics has the stronger property of definiteness.  Finally, an application is provided to the $m$-metric
  of point sets  when driven by the same linear ordinary differential equation.}




\maketitle

\section{Introduction}
\label{intro}
For the purpose of this paper we slightly deviate from the notion of an $m$-metric as established in \cite{Gaehler69}.
Given a set $X$ and a number $n \ge 2$, we call a map
\begin{align*} d:X^n=\prod_{i=1}^n X \to \R
\end{align*}
a {\it pseudo $n$-metric on $X$} if it has the following properties:
  \begin{itemize}
    \item[(\rm{$M_1$})] (Semidefiniteness) 
    If $x=(x_1,\ldots,x_n)\in X^n$ satisfies $x_i=x_j$ for some $i\neq j$, then
    $d(x_1,\ldots,x_n)=0$.
  \item[(\rm{$M_2$})] (Symmetry) 
    For all $x=(x_1,\ldots,x_n)\in X^n$ and all 
      $\pi\in \mathcal{P}_n$
      \begin{equation} \label{perm}
        d(x_{\pi(1)},\ldots,x_{\pi(n)})= d(x_1,\ldots,x_n).
      \end{equation}
      Here $\mathcal{P}_n$ denotes the group of permutations of $\{1,\ldots,n\}$.
    \item[(\rm{$M_3$})] (Simplex inequality)
      For all $x=(x_1,\ldots,x_n)\in X^n$ and $y \in X$
      \begin{equation} \label{tri}
        d(x_1,\ldots,x_n) \le \sum_{i=1}^n d(x_1,\ldots,x_{i-1},y,x_{i+1},\ldots,x_n).
      \end{equation}
      \end{itemize}
According to this definition a pseudo $2$-metric  agrees with the standard notion of an ordinary
pseudo (or quasi) metric. Then, as in case $n=2$,  we call a pseudo $n$-metric an $n$-metric if it is definite, i.e. 
 if $ d(x_1,\ldots,x_n) = 0$ implies $x_i=x_j$ for some $i \neq j$.
  Note that our definition of a   pseudo $n$-metric corresponds to a quasi $(m=n-1)$-metric in the sense of
  \cite[I, Ch.2.1]{Gaehler69}. There the quasi $m$-metric is called an $m$-metric
  if the following additional condition is satisfied for $n=m+1$
  \begin{itemize}
    \item[(\rm{$M_4$})]
      For all $i \neq j$ in $\{1,\ldots,n\}$ and  $x_i\neq x_j$ in $X$ there exist elements $x_k\in X$ for $k\neq i,j$
      such that $d(x_1,\ldots,x_n) \neq 0$.
  \end{itemize}
  This is a rather weak condition suitable to establish Hausdorff properties of the topology induced on $X$ ( see \cite[II, Satz 25]{Gaehler69}). We prefer to reserve the name  $n$-metric for a pseudo $n$-metric which is definite in the sense above. Let us further note  that the notion of an $n$-distance
  is used
  in the papers \cite{Marichal18,Marichal23} for functions $d$ which satisfy ($M_1$)-($M_3$) and
  which are called definite if $d(x_1,\ldots,x_n)=0$ implies  $x_1=\cdots = x_n$.

   The most common example of an $n$-metric is obtained when $d(x_1,\ldots,x_n)$  is taken as the volume
   of the simplex (or up to a factor of a parallelepiped) with vertices $x_1,\ldots,x_n$ in some linear space. For this case the metric properties ($M_1$)-($M_3$)
   were observed as early as 1928 by K. Menger in \cite{Menger28}. The value $d(x_1,\ldots,x_n)$ then vanishes
   if and only if the spanning vectors $x_2-x_1,\ldots, x_n- x_1$ are linearly dependent.
   This is just one example of  a standard construction of a  pseudo $n$-metric via a so-called $(m=n-1)$-normed
   space through $d(x_1,\ldots,x_n) = \|(x_2-x_1,\ldots,x_n-x_1)\|$; see \cite[I, Satz 5]{Gaehler69}. The theory of
   $m$-normed spaces has received considerable  attention. For example,
   the  isometry problem for $m$-norms is analyzed in \cite{huangtan18,huangtan19},
   applications to fixed point problems appear in 
   \cite{VVKR15} and several related papers cited therein, and generalizations to vector spaces over valued fields appear in \cite{Schwaiger23}. 
  Further, several geometric constructions leading to $n$-metrics (or $n$-distances) are presented  in \cite{Marichal18,Marichal20,Marichal23} with an emphasis on optimal constants for the simplex inequality.
   The papers \cite{Marichal18,Marichal23} also provide applications to $n$-distances in graphs
   based on spanning trees.
   
 In \cite{Be24} we re-discovered some of these constructions, but also detected some new pseudo $n$-metrics,
 for example  on the Grassmann and the Stiefel manifold, and on hypergraphs.
 There we also found the following $n$-metric in $\R$ which, in a sense, is the minimal expression
 leading to definiteness according to our definition  above:
 \begin{equation} \label{s1:e1}
   d_V(x_1,\ldots,x_n)= \prod_{1 \le j <i \le n} |x_i-x_j|.
 \end{equation}
 We call $d_V$ the Vandermonde $n$-metric since 
 $d_V(x_1,\ldots,x_n)=|V(x_1,\ldots,x_n)|$ holds with the Vandermonde determinant
 \begin{equation*} \label{s1:e2}
      V(x_1,\ldots,x_n) = \det \begin{pmatrix} 1 & \cdots & 1 \\ \vdots & \cdots & \vdots \\
   x_1^{n-1} & \cdots & x_n^{n-1} \end{pmatrix}.
 \end{equation*}
 Note that the Vandermonde $n$-metric does not derive from an $(n-1)$-norm in the way
 described above.
 
 The purpose of this contribution is to show that this expression has far-reaching generalizations
 which lead to a whole family of pseudo $n$-metrics in linear spaces of arbitrary dimension.

 In Section \ref{s2} we study the extension of \eqref{s1:e1} to the complex plane $\C$ where
 the  simplex inequality follows in a similar manner as in the real case \cite[Example 1]{Be24}.
 Despite its simplicity the simplex inequality for this case has  some interesting
 geometric consequences for cyclic $n$-gons, i.e. $n$-gons with vertices on a circle
 of radius $R$.
 We mention the following (see Corollaries  \ref{s2:cor1}, \ref{cor2:3-metric}
 and Figure  \ref{simplicial4})
 \begin{itemize}
 \item[I.] The sides $a,b,c$ of a triangle with  circumradius $R$ satisfy 
      \begin{equation} \label{s1:e3}
     abc \le R^2(a+b+c).
   \end{equation}
      Equality holds if and only if the triangle is equilateral.
    \item[II.] The sides $a,b,c,d$ and diagonals $e,f$ of a cyclic quadrangle
with circumradius  $R$  satisfy
      \begin{equation} \label{s1:e3a}
        abcdef \le R^3\left[(ab+cd)e + (ad+bc)f\right].
      \end{equation}
      Equality holds if and only if the quadrangle is equilateral.
       \item[III.] The expression
   \begin{equation} \label{s1:e4}
     d(x_1,x_2,x_3)= \|x_1-x_2\| \|x_1-x_3\| \|x_2-x_3\|, \quad x_1,x_2,x_3 \in X
   \end{equation}
   defines a $3$-metric in any Euclidean space $(X,\langle \cdot,\cdot \rangle, \| \cdot \|)$.
 \end{itemize}
 
 \begin{figure}[hbtp]
\begin{minipage}{0.4\textwidth}
\includegraphics[width=\textwidth]{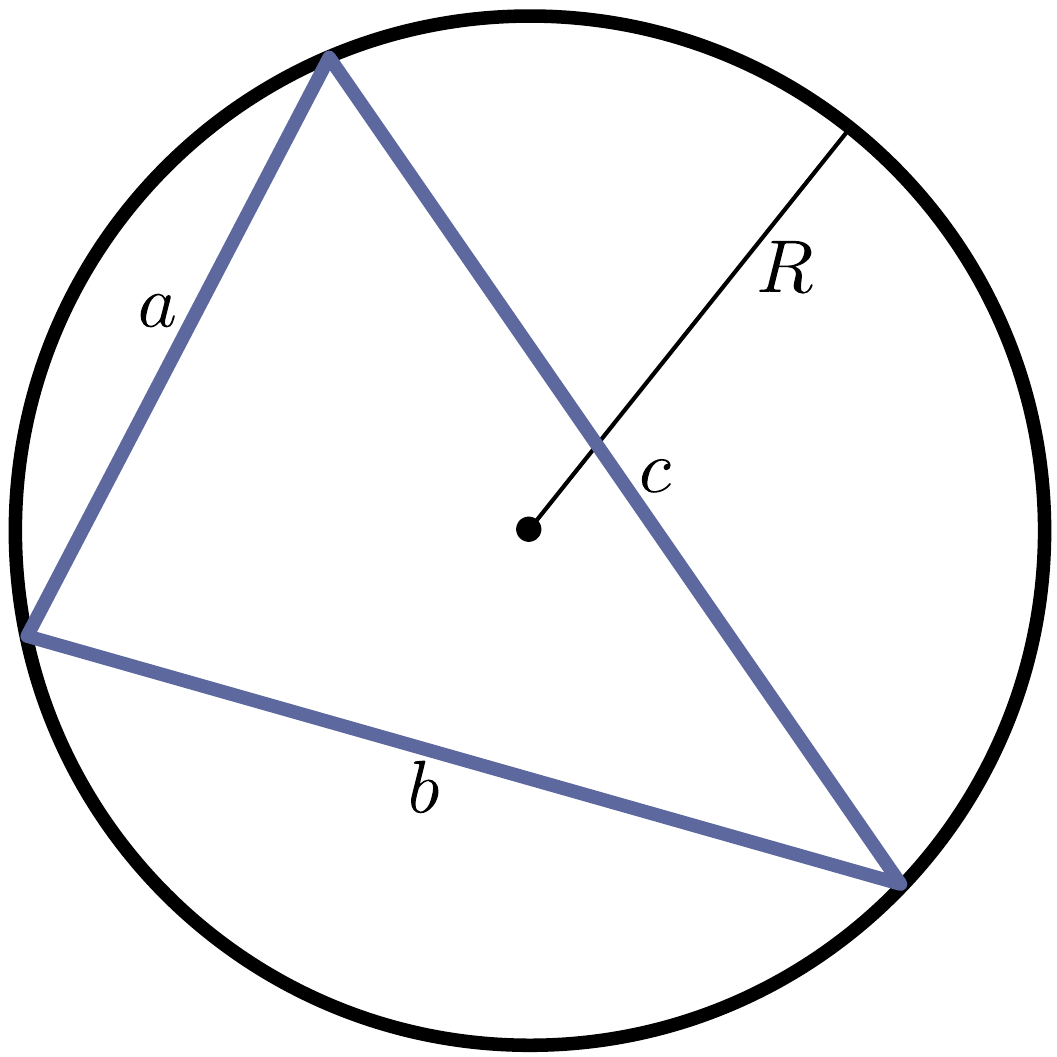}   
\end{minipage}
\hfill
\begin{minipage}{0.4\textwidth}
\includegraphics[width=\textwidth]{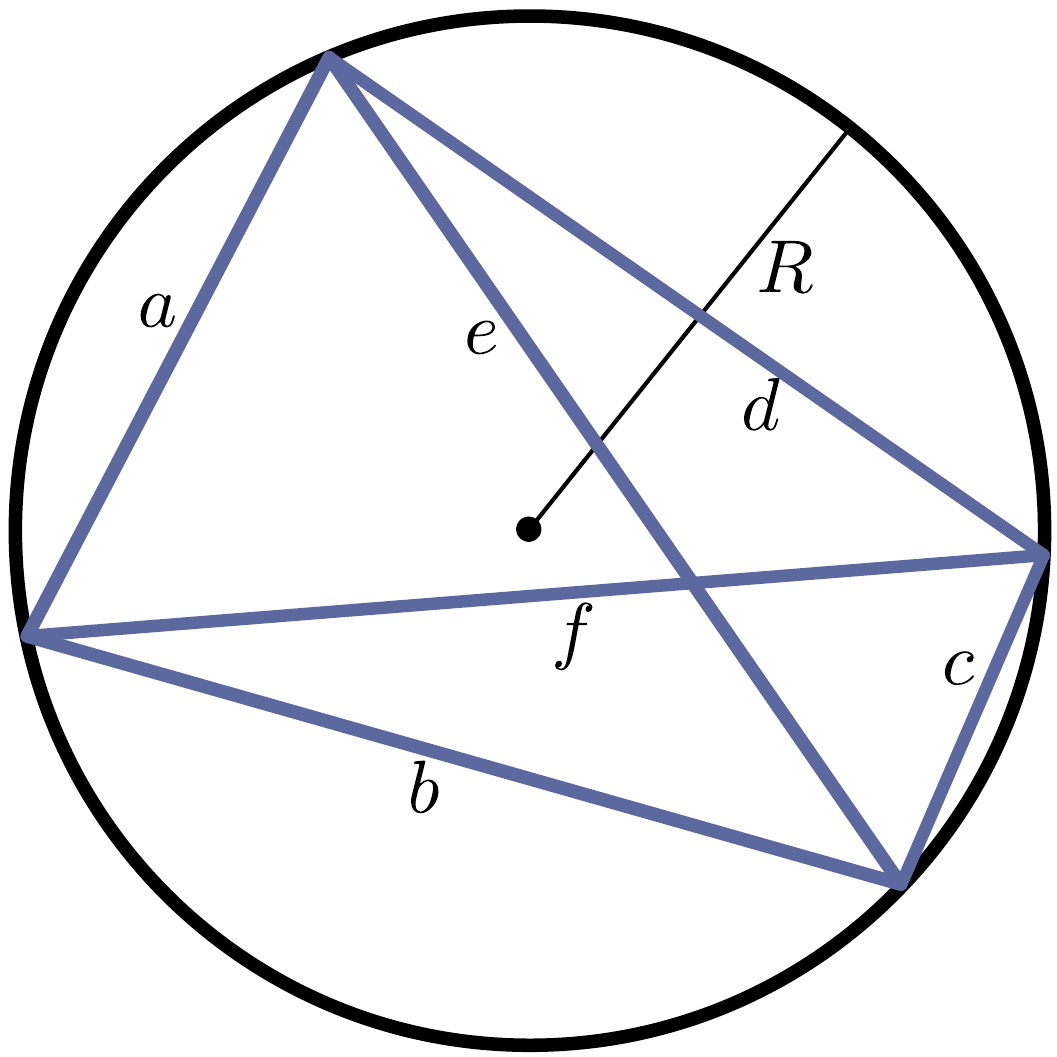}   
\end{minipage}
\caption{\label{simplicial4} The simplex inequality $abc \le R^2(a+b+c)$ for the Vandermonde $3$-metric (left) and $abdcef \le R^3(abe+cde+adf+bcf)$ for
  the Vandermonde $4$-metric (right).}
\end{figure}
 Inequality \eqref{s1:e3} is equivalent to Euler's inequality $2r \le R$ for the inradius $r$ by the familiar formula $2rR=\frac{abc}{a+b+c}$; see e.g. \cite{SV12}. To the best of
 our knowledge, inequality
 \eqref{s1:e3a} has gone unnoticed in the literature on estimates for cyclic quadilaterals. It will be of interest to find an elementary geometric proof  which
 avoids the use of complex numbers.
 
  In Section \ref{sec3}  we show that the Vandermonde $n$-metric can be generalized to arbitrary vector spaces
  $X$ by the formula
  \begin{equation} \label{s1:e5}
    d(x_1,\ldots,x_n) = \left\|A\left( (x_i-x_j)_{1 \le j < i \le n}\right)\right\|,
  \end{equation}
  where $A$ is a $\frac{1}{2}n(n-1)$-linear and symmetric form on $X$ with values in some
    normed space $(Y,\|\cdot\|)$.
    While definiteness is obvious for the scalar case \eqref{s1:e1},  it may be lost in the general setting
   of  \eqref{s1:e5}. We discuss
    a particular multilinear map $A$ for which definiteness is preserved.

    Finally, we emphasize that, adding to its geometric and topological features, an $n$-metric may also
    play a role as a measure of the point cluster $\{x_1,\ldots,x_n\}$, in particular
    when the points are driven by a dynamical system.  A simple estimate supporting this
    viewpoint is provided in Section \ref{sec3.4}.


\section{ The Vandermonde $n$-metric in $\C$}\label{s2}
 In this section we consider the basic construction of an $n$-metric in $\C$ via the well-known Vandermonde determinant associated with numbers $z_j \in \C$, $j=1,\ldots,n$ (see \cite[Ch.0.9]{HJ2013}):
\begin{equation} \label{CdefVandermond}
  V(z_1,\ldots,z_n) = \det \begin{pmatrix}
    1 & \cdots & 1 \\
    \vdots & \cdots & \vdots \\
    z_1^{n-1} & \cdots & z_n^{n-1} \end{pmatrix}=
  \prod_{1 \le j< i \le n} (z_i-z_j).
\end{equation}
\begin{theorem}(Vandermonde $n$-metric) \label{Cprop}
    The expression
    \begin{equation} \label{Cprodn}
      d_V(z_1,\ldots,z_n) = \prod_{1\le j< i \le n}|z_i-z_j|
      =|V(z_1,\ldots,z_n)|
    \end{equation}
    defines an $n$-metric in $\C$ and by restriction also in $\R$.
         \end{theorem}
\begin{proof} According to Proposition \ref{prop1.1}(i) below it suffices to consider \eqref{Cprodn}
  in $\C$.
  By the definition \eqref{Cprodn} it is clear that $d_V$ is nonnegative, semidefinite, and even definite. The symmmetry is a consequence of the following equality for $\pi \in \mathcal{P}_n$:
  \begin{equation*}
    d_V(z_1,\ldots,z_n)^2 = \prod_{i \neq j}|z_i-z_j|= \prod_{\pi(i)\neq \pi(j)}
    |z_{\pi(i)}-z_{\pi(j)}|= d_V(z_{\pi(1)},\ldots,z_{\pi(n)})^2.
    \end{equation*}
  First note that it  suffices to show the simplex inequality for pairwise different numbers $z_1,\ldots,z_n$ (otherwise the left-hand side of \eqref{tri} vanishes). Then there is a unique
    solution $(a_1,\ldots,a_n)^{\top}\in \C^n$ to the linear system
    \begin{align} \label{s1:linsyst}
\begin{pmatrix}
    1 & \cdots & 1 \\
    \vdots & \cdots & \vdots \\
    z_1^{n-1} & \cdots & z_n^{n-1} \end{pmatrix} \begin{pmatrix}a_1 \\ \vdots \\ a_n
\end{pmatrix} = \begin{pmatrix}1 \\ \vdots \\  y^{n-1} \end{pmatrix}.
    \end{align}
    The solution is  given by Cramer's rule as follows
    \begin{align}\label{s2:Cramer}
  a_i = \frac{V(z_1,\ldots,z_{i-1},y,z_{i+1},\ldots,z_n)}{V(z_1,\ldots,z_n)},
  \quad i=1,\ldots,n,
    \end{align}
    with the Vandermonde determinant from \eqref{CdefVandermond}.
    Using this we estimate
    \begin{equation} \label{Vest}
    \begin{aligned}
     & d_V(z_1,\ldots,z_n) = |V(z_1,\ldots,z_n)\sum_{i=1}^n a_i|
       = |\sum_{i=1}^n V(z_1,\ldots,z_{i-1},y,z_{i+1},\ldots,n)|\\
      & \quad \le \sum_{i=1}^n |V(z_1,\ldots,z_{i-1},y,z_{i+1},\ldots,n)|
      = \sum_{i=1}^n d_V(z_1,\ldots,z_{i-1},y,z_{i+1},\ldots,n),
    \end{aligned}
    \end{equation}
    which completes the proof.
     \end{proof}
\begin{remark}
  The  proof yields the following more general
  inequality
  \begin{align} \label{eq2:gsi}
    |y|^k d_V(z_1,\ldots,z_n) \le \sum_{i=1}^n |z_i|^k d_V(z_1,\ldots,z_{i-1},y,z_{i+1},
    \ldots,z_n), \; k=0,\ldots,n-1.
  \end{align}
  Another observation is that $d_V$ is invariant under shifts and positive homogeneous of degree $\frac{n(n-1)}{2}$ with respect to multiplication
    by a complex number $y \in \C$
    \begin{equation*} \label{eq2:dVprop}
      \begin{aligned}
        d_V(z_1,\ldots,z_n)& = d_V(z_1-y,\ldots, z_n-y), \\
        d_V(y z_1,\ldots,y z_n)& = |y|^{\frac{n(n-1)}{2}}d_V(z_1,\ldots,z_n).
      \end{aligned}
    \end{equation*}
    Further, note that we obtain an $n$-metric  which is homogeneous of degree $1$ by setting
    \begin{equation} \label{eq2:rootmetric}
      d_*(z_1,\ldots,z_n)=d_V(z_1,\ldots,z_n)^{\frac{2}{n(n-1)}},
    \end{equation}
     since  the map $\lambda>0 \mapsto \lambda^q$ is sublinear for $0\le q \le 1$.
\end{remark}
In case $n=3$ we investigate when the simplex inequality holds with equality.
    \begin{proposition} \label{corcomplex}
      For the $3$-metric
    \begin{equation*} \label{Cprodn=3}
      d_V(z_1,z_2,z_3) = |z_1-z_2||z_2-z_3||z_1 -z_3|, \quad z_1,z_2,z_3\in \C
    \end{equation*}
   the equality
    \begin{align} \label{Csimpequality}
      d_V(z_1,z_2,z_3) = d_V(y,z_2,z_3)+d_V(z_1,y,z_3)+d_V(z_1,z_2,y)
    \end{align}
    holds with pairwise different numbers $y, z_1,z_2,z_3 \in \C$ iff the 
    quadruple $(y,z_1,z_2,z_3)$ belongs (up to a shift and a multiplication
    by a complex number and up to a permutation of $z_1,z_2,z_3$) to the following two parameter family $(q,s>0)$:
    \begin{equation} \label{Ceqfamily}
     y=0, \quad z_1=1, \quad z_2=\frac{1}{s}\Big(-1 + i \sqrt{q(1+s)}\Big),
      \quad z_3=\frac{1}{s}\Big(-1 - i \sqrt{q^{-1}(1+s)}\Big).
    \end{equation}
    Further, the equality \eqref{Csimpequality} holds with $|z_1-y|=|z_2-y|=
    |z_3-y|$  if and only if the triangle $z_1,z_2,z_3$ is equilateral. 
    \end{proposition}
    \begin{proof}
     From \eqref{Vest} we find that equality 
    holds in \eqref{tri} if and only if
    \begin{equation} \label{CeqV}
      \big| \sum_{i=1}^n V(z_1,\ldots,z_{i-1},y,z_{i+1},\ldots,z_n) \big|
      = \sum_{i=1}^n |V(z_1,\ldots,z_{i-1},y,z_{i+1},\ldots,z_n)|.
    \end{equation}
    Since the numbers are distinct we may shift $y$ to zero and multiply
    by a complex number such that $z_1=1$.
   Recall that the equality $|z_1|+|z_2|=|z_1+z_2|$ holds for $z_1,z_2\in \C$
    if and only if either $z_1=cz_2$ or $z_2=cz_1$ for some $c \ge 0$.
    More generally, the equality $|\sum_{i=1}^n z_i|=\sum_{i=1}^n |z_i|$
    holds if and only if there exists an index $j \in \{1,\ldots,n\}$
    and real numbers $c_i\ge 0$ such that $z_i = c_i z_j$ for $i=1,\ldots,n$.
    Moreover, if $z_i \neq 0$ for all $i=1,\ldots,n$ then the latter
    property holds for any $j\in\{1,\ldots,n\}$ with numbers $c_i>0$.
    We apply this to \eqref{CeqV} with $n=3$ after normalizing $y=0$
    and $z_1=1$. Thus the equality \eqref{Csimpequality} holds if and only
    if there are numbers $c_2,c_3>0$ such that
    \begin{equation*} \label{Csyst1}
      \begin{aligned}
        V(1,0,z_3)& = (-z_3)(1-z_3) = c_2 V(0,z_2,z_3)=c_2(-z_2)(-z_3)(z_2-z_3),\\     V(1,z_2,0)& = (1-z_2)z_2= c_3V(0,z_2,z_3)= c_3 (-z_2)(-z_3)(z_2-z_3).
      \end{aligned}
    \end{equation*}
    Since $z_2,z_3 \neq 0$ this is equivalent to the system
    \begin{equation*} \label{Csyst2}
      \begin{aligned}
        1-z_3& = - c_2z_2(z_2-z_3), \\
        1-z_2& = c_3 z_3(z_2-z_3).
      \end{aligned}
    \end{equation*}
    Subtract the second from the first equation and use $z_2 \neq z_3$
   to find the equivalent system
    \begin{equation} \label{Csyst3}
      \begin{aligned}
        1-z_3& = - c_2z_2(z_2-z_3), \\
        1& = -c_2z_2 -c_3 z_3. 
      \end{aligned}
    \end{equation}
    With the last equation we eliminate $z_3=-c_3^{-1}(1+c_2z_2)$ from
    the first  to obtain the quadratic equation
    \begin{align*}
      c_2(c_2+c_3)z_2^2+ 2 c_2 z_2 +c_3 + 1 =0.
    \end{align*}
    The solutions to \eqref{Csyst3} are then given by
    \begin{equation*}
      \begin{aligned}
        z_2^{\pm}&=\frac{1}{c_2+c_3}\Big(-1 \pm i \big( \frac{c_3}{c_2}(1+c_2+c_3)\big)^{1/2}\Big),\\
        z_3^{\mp}&
        =\frac{1}{c_2+c_3}\Big(-1 \mp i \big( \frac{c_2}{c_3}(1+c_2+c_3)\big)^{1/2}\Big).
      \end{aligned}
    \end{equation*}
    Introducing the parameters $s=c_2+c_3 >0$ and $q=\frac{c_3}{c_2}>0$
    we may write this as
    \begin{equation*}
      z_2^{\pm}(q,s)=\frac{1}{s}\Big(-1 \pm i \sqrt{ q(1+s)}\Big),
      \quad
        z_3^{\mp}(q,s)
        =\frac{1}{s}\Big(-1 \mp i \sqrt{ q^{-1}(1+s)}\Big).
    \end{equation*}
    Finally, we observe the relation
    \begin{align*}
      (z_2^{-},z_3^+)(q^{-1},s)=(z_3^-,z_2^+)(q,s), \quad q,s>0,
    \end{align*}
    where the latter is a permutation of $(z_2^+,z_3^-)(q,s)$.
    Hence the family \eqref{Ceqfamily} covers all solutions up to permutation.\\
    For the normalized coordinates $y=0, z_1=1$ the condition $1=|z_2|=|z_3|$
    holds if and only if $s=2,q=1$. For these values the numbers
    $z_2=\frac{1}{2}(-1 + i \sqrt{3})$, $z_3=\frac{1}{2}(-1 - i \sqrt{3})$
      are the third roots of unity which belong to the equilateral case.
      \end{proof}

    An easy consequence of Theorem \ref{Cprop} and Proposition \ref{corcomplex} are some elementary geometric inequalities for cyclic polygons 
   (see Figure  \ref{simplicial4}).
    \begin{corollary} \label{s2:cor1} In the following let $R$ be the circumradius of
   a    cyclic $n$-gon.
      \begin{itemize}
    \item[(i)] A triangle with sides $a$, $b$, $c$ and circumradius $R$ satisfies
            \begin{equation} \label{s2:si3}
        abc \le R^2(a+b + c).
        \end{equation}
            Equality holds if and only if the triangle is equilateral.
            \item[(ii)] The sides $a$, $b$, $c$, $d$ and diagonals $e$, $f$
            of a cyclic quadrangle satisfy
            \begin{equation} \label{s2:si4}
              abcdef \le R^3(abe + bcf + cde + adf).
            \end{equation}
            Equality holds if and only if the quadrangle is equilateral.
          \item[(iii)] The vertices $z_j$, $j=1,\ldots,n$ of a cyclic $n$-gon ($n\ge 3$)
            with circumradius $R$ satisfy
            \begin{equation} \label{s2:sin}
              \prod_{1 \le j < i \le n}|z_i-z_j| \le R^{\frac{(n+1)(n-2)}{2}} (n-2)! \sum_{1\le j < i \le n} |z_i-z_j|.
              \end{equation}
          \item[(iv)] For the cyclic  $n$-gon equality holds in  the simplex
            inequality for the Vandermonde $n$-metric $d_V$ from \eqref{Cprodn} if and only if the $n$-gon is equilateral.
      \end{itemize}
    \end{corollary}
    \begin{proof}
      Inequalities \eqref{s2:si3} and \eqref{s2:si4}  follow from the simplex inequality for $d_V$ with $n=3,4$ by taking $y$ as  the circumcenter of the circumcircle. The characterization of  equality follows for $n=3$ from Proposition \ref{corcomplex}
      and will be proved for general $n$ in assertion (iv).
    The proof of (iii) proceeds by induction. In case $n=3$ inequality \eqref{s2:sin} follows from \eqref{s2:si3}.
    For the  induction step from $n-1\ge 2$ to $n$ we use the simplex inequality with $y$ as circumcenter
    \begin{align*}
      \prod_{1 \le j < i \le n}|z_i-z_j| & \le \sum_{k=1}^n \prod_{1\le j \le n, j\neq k}|z_j-y| \prod_{1\le j <i \le n, i,j \neq k} |z_i -z_j|\\
      & \le R^{n-1} (n-3)!  R^{\frac{n(n-3)}{2}}\sum_{1 \le j < i \le n} \sum_{k\in \{1,\ldots,n\}\setminus \{i,j\}} |z_i -z_j| \\
         & = (n-2)! R^{\frac{(n+1)(n-2)}{2}}\sum_{1 \le j < i \le n}|z_i-z_j|.
    \end{align*}
    For the  last step note that there are $n-2$ numbers
    $k \in \{1,\ldots,n\}$ satisfying $k \neq i,j$ for a given a pair $1\le j < i \le n$.

    To prove assertion (iv)  we first show that the solution of \eqref{s1:linsyst} is $a_j=\frac{1}{n}$,
    $j=1,\ldots,n$ if $y=0$ and
    $z_j =z \omega_n^{j-1}$, $\omega_n= \exp(\frac{2 \pi i}{n})$, $|z|=1$. In fact, we have $\sum_{j=1}^n a_j =1$ and for $k=2,\ldots,n$
    \begin{equation*}
       \sum_{j=1}^n z_j^{k-1}=z^{k-1} \sum_{j=1}^n (\omega_n^{k-1})^{j-1}= z^{k-1} \frac{\omega_n^{(k-1)n}-1}{\omega_n^{k-1}-1}=0.
    \end{equation*}
    The solution formula \eqref{s2:Cramer} then shows that we have equality in \eqref{Vest}.
    
    Conversely, assume that equality holds in \eqref{Vest} for pairwise different points $|z_k|=R$, $k=1,\ldots,n$ and $y=0$. By scaling
    with $R^{n(n-1)/2}$ we can assume $R=1$ without loss of generality. Moreover, by the symmetry of $d_V$ we can order the
    points by increasing arguments, i.e. $z_k= e^{i \varphi_k}$, $k=1,\ldots,n$  where $0 \le \varphi_1< \cdots < \varphi_n < 2 \pi$.
    
    As noted in the proof of Proposition \ref{corcomplex}, the  equality in \eqref{Vest} implies $a_k>0$, $k=1,\ldots,n$ for the solution of \eqref{s1:linsyst}
    with $y=0$ on the right-hand side.
       Next we consider Lagrange interpolation of some function $f:\C \to \C$
    \begin{equation} \label{s2:Lagrange}
      p_f(z) = \sum_{k=1}^n f(z_k) L_k(z), \quad L_k(z) = \prod_{\ell=1, \ell \neq k}^n \frac{z_{\ell}-z}{z_{\ell}-z_k}.
    \end{equation}
    The $k$-th Lagrange polynomial is of the form $L_k(z) = \sum_{j=1}^n \alpha_j z^{j-1}$ where $\alpha \in \C^n$ solves
the linear system
    \begin{equation*}
      \alpha^{\top} \begin{pmatrix} 1 & \cdots & 1 \\
        \vdots & \cdots & \vdots \\
        z_1^{n-1} & \cdots & z_n^{n-1} \end{pmatrix} = (e^k)^{\top}=
        \begin{pmatrix} 0 & \cdots& 0 &1 & 0 & \cdots 0
      \end{pmatrix}. 
    \end{equation*}
    Multiplying from the right by $a\in \R^n$ we obtain $a_k=(e^k)^{\top}a= \alpha^{\top} e^1= \alpha_1=L_k(0)$.
    Now we interpolate $f(z)=z^n$ and use the well-known error representation 
    \begin{equation*}
      f(z) - p_f(z) = \prod_{k=1}^n (z-z_k), \quad z \in \C.
    \end{equation*}
    Recall that both sides are polynomials which  vanish at $z_j$, $j=1,\ldots,n$ and have leading term $z^n$,
    hence coincide.
    Then we find
    \begin{equation*}
      p_f(0)=\sum_{k=1}^n z_k^n L_k(0)= \sum_{k=1}^n a_k z_k^n = f(0)- \prod_{k=1}^n(-z_k)=(-1)^{n+1} \prod_{k=1}^n z_k.
    \end{equation*}
    Using $a_k>0$ and $|z_k|=1$ we obtain the equality 
    \begin{equation*}
      \sum_{k=1}^n |a_k z_k^n| = \sum_{k=1}^n a_k =1 = \Big| (-1)^{n+1}\prod_{k=1}^n z_k \Big|=
      \Big| \sum_{k=1}^n a_k z_k^n\Big|.
    \end{equation*}
    As in Proposition \ref{corcomplex} we conclude $z_k^n=c_k z_1^n$ for some $c_k>0$, $k=1,\ldots,n$.
    Since $|z_k^n|=1$, we even have $c_k=1$ for all $k=1,\dots,n$.
    Thus we obtain  $(z_1^{-1}z_k)^n=1$ for $k=1,\ldots,n$; i.e. the
    numbers $z_1^{-1}z_k$, $k=1,\ldots,n$ are $n$ different $n$-th roots of unity. By the ordering of their arguments
     we obtain  $z_1^{-1}z_k=e^{i(\varphi_k-\varphi_1)}= \omega_n^{k-1}$ for $k=1,\ldots,n$ . 
    Hence, the numbers $z_k=z_1 \omega_n^{k-1}$, $k=1,\ldots,n$ are the vertices of an  equilateral $n$-gon.
    \end{proof}
    

Another implication of Theorem \ref{Cprop} in case $n=3$ is the following:

\begin{corollary} \label{cor2:3-metric}
  Let $(X,\langle \cdot,\cdot \rangle,\|\cdot\|)$ be a Euclidean vector space. Then
  \begin{equation*}
    d_3(x_1,x_2,x_3) = \|x_1-x_2\| \|x_1-x_3\| \|x_2 - x_3\|, \quad x_1,x_2,x_3 \in X
  \end{equation*}
  defines a $3$-metric on $X$.
  \end{corollary}
\begin{proof}
  The symmetry and definiteness are obvious from the definition. We prove the simplex inequality.
  By the shift invariance of $d_3$ we can assume $x_1=0$.
  The points $x_1,x_2,x_3$ then lie in the subspace $H=\mathrm{span}(x_2,x_3)$ where
  $\dim(H)=h\le 2$. Choose an orthonormal basis in $H$, i.e. a
  map $Q \in L(\R^h,X)$ which is unitary (i.e. $Q^{\star}Q=I_h$)  and satisfies $H =\mathrm{range}(Q)$.
  Then we have $x_j=Qz_j$ $j=1,2,3$ for $z_1=0$ and  some $z_2,z_3 \in \R^h$.
  Let $y\in X$ be arbitrary and let $Py= Qz$ with $z= Q^{\star}y$ be its orthogonal projection onto
  $H$.  From  the simplex inequality in $\R^h$ we obtain
  \begin{align*}
    &\|x_2\| \|x_3\| \|x_3-x_2\| = \|z_2\| \|z_3\| \|z_3-z_2\|\\
    & \le  \|z_2-z\| \|z_3-z\| \|z_3-z_2\|+  \|z\| \|z_3\| \|z_3-z\| 
     + \|z_2\| \|z\| \|z-z_2\|.
  \end{align*}
    Since by Pythagoras
  \begin{align*}
    \|z_j-z\|^2&=\|Qz_j-Qz\|^2 = \|x_j - Py\|^2 \le \|x_j-Py\|^2+ \|Py -y\|^2= \|x_j-y\|^2
    \end{align*}
  holds for $j=1,2,3$, we obtain the conclusion
   \begin{align*}
  d_3(0,x_2,x_3)&=  \|x_2\| \|x_3\| \|x_3-x_2\| \\ 
    &\le  \|x_2-y\| \|x_3-y\| \|x_3-x_2\| +  \|y\| \|x_3\| \|x_3-y\| 
     + \|x_2\| \|y\| \|y-x_2\| \\
    &= d_3(y,x_2,x_3)+d_3(0,y,x_3) +d_3(0,x_2,y).
   \end{align*}
   \flushright \end{proof}

 It is tempting to conjecture that 
 \begin{equation} \label{gendef}
  d(x_1,\ldots,x_n) = \prod_{1 \le j < i \le n} \|x_i-x_j\|,
  \quad x_i\in X, i=1,\ldots, n,
  \end{equation}
 generates an $n$-metric on a Euclidean space.
 However, the expression \eqref{gendef} does not satisfy the simplex inequality
 for $n=4$ 
    in dimensions $\ge 3$ as the following example shows.
 \begin{example} \label{s2:ex1}
   Consider $X=\R^3$, $y=0$, and let $x_1,\ldots,x_4$ be the vertices
    of an equilateral tetrahedron:
    \begin{align*}
      x_1&=(1,0,0), \quad x_2 = \frac{1}{3}(-1,2\sqrt{2},0), \\
      x_3& =\frac{1}{3}(-1,-\sqrt{2},\sqrt{6}), \quad
      x_4= \frac{1}{3}(-1,-\sqrt{2}, -\sqrt{6}).
    \end{align*}
    Then one verifies
    \begin{align*}
      \|x_j\|&=1 , \quad j=1,\ldots,4, \\
      \|x_i -x_j\|& =  \sqrt{\frac{8}{3}}, \quad 1\le i < j \le 4.
    \end{align*}
    The simplex inequality for \eqref{gendef} requires
    \begin{align*}
      \big( \frac{8}{3}\big)^3 \le 4 \big( \frac{8}{3}\big)^{3/2}
      \quad \text{or} \quad 2^5 \le 3^3,
    \end{align*}
    which is wrong. However, note that this counterexample does not work for the modified expression
    $d_{\star}(x_1,\ldots,x_4)= d(x_1,\ldots,x_4)^{1/6}$ suggested by \eqref{eq2:rootmetric}. 
 \end{example}

  In  Section \ref{sec3} we present an alternative generalization of the
  Vandermonde pseudo $n$-metric which works in arbitrary dimensions but
  is not necessarily definite.

  The following proposition collects basic properties of pseudo $n$-metrics
  which allow us to transfer the Vandermonde metric $d_V$ to some function spaces.

\begin{proposition} \label{prop1.1}
  For a pseudo $n$-metric $d_X$ on a set $X$ the following holds:
  \begin{itemize}
  \item[(i)] $d_X$ defines a pseudo $n$-metric on any subset $Y\subset X$.
  \item [(ii)] If $d_Y$ is a pseudo $n$-metric an a set $Y$, then
    \begin{equation*} \label{eqprod}
      d_{X \times Y}\Big( \begin{pmatrix} x_1 \\ y_1 \end{pmatrix}, \ldots,
      \begin{pmatrix} x_n \\ y_n \end{pmatrix} \Big) :=
      \Big\| \begin{pmatrix} d_X(x_1,\ldots,x_n) \\ d_Y(y_1, \ldots,y_n) \end{pmatrix} \Big\|
    \end{equation*}
    defines a pseudo $n$-metric on $X \times Y$ for any monotone norm $\|\cdot\|$
    in $\R^2$.
    \item[(iii)] Let $\cF \subseteq X^Y$ be a subset of functions from some
set $Y$ to $X$. Further assume that a normed space
$Z \subseteq \R^{Y}$ is given with a monotone norm $\|\cdot \|_Z$ and
the following property:
\begin{align} \label{Fprop}
  f_1,\ldots,f_n \in \cF \Longrightarrow d(f_1(\cdot),\ldots,f_n(\cdot))
  \in Z.
\end{align}
Then a pseudo $n$-metric on $\cF$ is defined by
\begin{equation*} \label{pseudoF}
  d_{\cF}(f_1,\ldots,f_n)  = \| d(f_1(\cdot),\ldots,f_n(\cdot)) \|_Z.
\end{equation*}
 \end{itemize}
\end{proposition}
\begin{proof}  Assertions (i) and (ii) are rather obvious; cf. \cite[Proposition 2]{Be24}.
   Recall that a norm $\|\cdot\|_Z$
  on a linear space $Z$ of real-valued functions $f:Y \to \R$
  is called monotone (see e.g. \cite[Ch.5.4]{HJ2013}) iff for all $f_1,f_2 \in Z$
  \begin{align*}
    |f_1(y)| \le |f_2(y)| \; \forall y \in Y \Longrightarrow
    \|f_1\|_Z \le \|f_2\|_Z.
  \end{align*}
 For the proof of (iii) let $f_1,\ldots,f_n,g \in \cF$ be given and  observe the pointwise inequality
\begin{align*}
  d(f_1(y),\ldots,f_n(y)) & \le \sum_{i=1}^n d(f_1(y),\ldots,f_{i-1}(y),g(y),
  f_{i+1}(y),\ldots f_n(y)), \quad \forall y \in Y.
\end{align*}
Then  the monotonicity of the norm implies
\begin{align*}
  d_{\cF}(f_1,\ldots,f_n) & \le \sum_{i=1}^n \| d(f_1(\cdot),\ldots,f_{i-1}(\cdot),g(\cdot),f_{i+1}(\cdot),\ldots,f_n(\cdot)) \|_Z \\
  & = \sum_{i=1}^n d_{\cF}(f_1,\ldots,f_{i-1},g,f_{i+1},\ldots, f_n).
\end{align*}
\flushright \end{proof}

    Combining Proposition \ref{prop1.1} and Theorem \ref{Cprop} leads to
    the following examples of pseudo $n$-metrics.
    \begin{example} \label{ex1}
      Let $\|\cdot\|$ be any monotone norm in $\R^k$. Then
      \begin{align*}
        d_k(x_1,\ldots,x_n)= \left\| \begin{pmatrix} d_V(x_{1,1},\ldots,x_{n,1})\\
          \vdots \\ d_V(x_{1,k},\ldots,x_{n,k}) \end{pmatrix} \right\|,
        \quad x_j = (x_{j,i})_{i=1}^k \in \R^k
      \end{align*}
      defines a pseudo $n$-metric in $\R^k$. Note that $d_k$, $k \ge 2$ is
      not definite although $d_V$ is.
    \end{example}
    \begin{example} \label{ex2}
      Let $(\Omega,\cG,\mu)$ be  a bounded measure space. We apply Proposition
      \ref{prop1.1}
      to  $Z=L^p(\Omega,\cG;\R)$ with the monotone norm
      $\|\cdot\|_{L^p}$  and to $\cF = L^r(\Omega,\cG;\R)$ with  $2r \ge n(n-1)p$.
      Then the expression
\begin{align*}
  d_{L^r}(f_1,\ldots,f_n)& = \Big( \int_{\Omega} \prod_{1\le j < i \le n}|f_i- f_j|^p
  \, \mathrm{d}\mu \Big)^{1/p}.
\end{align*}
defines a pseudo $n$-metric on $L^r(\Omega,\cG;\R)$. Note that $2r \ge n(n-1)p$ guarantees
that the property \eqref{Fprop} holds.
      \end{example}

 \section{Generalized Vandermonde pseudo $n$-metrics}
 \label{sec3}
 In the following we generalize   the Vandermonde expression  \eqref{Cprodn}
to arbitrary dimensions by using  multilinear symmetric maps.
 \subsection{ A formula for symmetric multilinear maps}
 \label{sec3.1}
Our starting point is the  following purely algebraic result.
\begin{theorem} (The general Vandermonde equalities) \label{th:gV}\\
    Let $X,Y$ be vector spaces over $\K=\R,\C$ and let
    \begin{align*}
      A\colon X^{M_n}\longrightarrow Y, \quad M_n = \frac{1}{2}n(n-1), n \ge 2,
    \end{align*}
    be an $M_n$-linear and symmetric map. 
    \begin{itemize}
    \item[(i)] For all $x_1,\ldots,x_n \in X$ the following holds:
      \begin{equation} \label{eq:expandV}
        A\Big(\prod_{1\le j < i \le n}(x_i-x_j)\Big)=
        \sum_{\pi\in \mathcal{P}_n} \mathrm{sign}(\pi) A\Big( \prod_{j=1}^n
        x_j^{\pi(j)-1}\Big).
      \end{equation}
    \item[(ii)] The map $V: X^n \to Y$  defined by
      \begin{align*}
        V(x_1,\ldots,x_n) = A\Big(\prod_{1\le j < i \le n}(x_i-x_j)\Big),
      \end{align*}
     satisfies for all $y \in X$ and for all $x_1,\ldots,x_n \in X$ the equality:
      \begin{equation} \label{eq:sumV}
        V(x_1,\ldots,x_n) = \sum_{i=1}^n V(x_1,\ldots,x_{i-1},y,x_{i+1},\ldots,x_n).
      \end{equation}
    \end{itemize}
\end{theorem}
\begin{proof} First note that it is convenient to write the arguments of the symmetric map $A$ as products
  since the sequence of arguments does not matter. With this notation, a term of type $x_j^0$ vanishes. \\
    {\bf (i):} \\
  For $n=2$ the equation \eqref{eq:expandV} is obvious:
  \begin{align*}
    \mathrm{sign}(\mathrm{id}) A(x_2^{\pi(2)-1})+ \mathrm{sign}(1,2)A(x_1^{\pi(1)-1})
    = A x_2 - A x_1 = A(x_2-x_1).
  \end{align*}
  For the induction step from $n$ to $n+1$ we use the symmetry of $A$:
\begin{equation} \label{eq:bigsum}
  \begin{aligned}
    & A\Big(\prod_{1\le j < i \le n+1}(x_i-x_j)\Big) = A\Big(\prod_{1\le j < i \le n}(x_i-x_j),\prod_{j=1}^n(x_{n+1}-x_j)\Big)\\
    & = \sum_{\pi \in \mathcal{P}_n}\mathrm{sign}(\pi)A\Big(\prod_{j=1}^n x_j^{\pi(j)-1},\prod_{j=1}^n(x_{n+1}-x_j)\Big)\\
    & = \sum_{\pi \in \mathcal{P}_n}\mathrm{sign}(\pi) \sum_{\sigma \subseteq
      \{1,\ldots,n\}}(-1)^{|\sigma|}A\Big(\prod_{j=1}^n x_j^{\pi(j)-1},
    \prod_{j \in \sigma}x_j ,x_{n+1}^{n - |\sigma|} \Big).
  \end{aligned}
  \end{equation}
  We consider the summand for any two indices $\ell,k\in \{1,\ldots,n\}$ with
  $\pi(\ell)= \pi(k)+1$ (hence $\ell \neq k$)  and any
  $\sigma \subseteq\{1,\ldots,n\}$ with $k \in \sigma, \ell \notin \sigma$.
  Let $\tau_{k,\ell}$ be the transposition of $k$ and $\ell$ and set
  \begin{align*}
    \tilde{\pi} & = \pi \circ \tau_{k,l}, \quad \tilde{\sigma}=(\sigma \setminus
    \{k\})\cup \{\ell\}.
  \end{align*}
  Then we obtain $|\tilde{\sigma}|=|\sigma|$, $\mathrm{sign}(\tilde{\pi})=
  - \mathrm{sign}(\pi)$, $\tilde{\pi}(k)=\pi(\ell)$, $\tilde{\pi}(\ell)=\pi(k)$ and
 
  \begin{align*}
    & \mathrm{sign}(\tilde{\pi})(-1)^{|\tilde{\sigma}|}A\Big(\prod_{j=1}^n x_j^{\tilde{\pi}(j)-1},
    \prod_{j \in \tilde{\sigma}}x_j ,x_{n+1}^{n - |\tilde{\sigma}|} \Big)\\
    & = -\mathrm{sign}(\pi)(-1)^{|\sigma|} A\Big( \prod_{j=1,j\neq k,\ell}^n
    x_j^{\pi(j)-1}, x_k^{\pi(\ell)-1}, x_{\ell}^{\pi(k)-1}, x_{\ell},
    \prod_{j\in \sigma,j\neq k}x_j, x_{n+1}^{n - |\sigma|}\Big)\\
    &  = -\mathrm{sign}(\pi)(-1)^{|\sigma|} A\Big( \prod_{j=1,j\neq k,\ell}^n
    x_j^{\pi(j)-1}, x_k^{\pi(k)}, x_{\ell}^{\pi(\ell)-1},
    \prod_{j\in \sigma,j\neq k}x_j, x_{n+1}^{n - |\sigma|}\Big)\\
    & = -\mathrm{sign}(\pi)(-1)^{|\sigma|} A\Big( \prod_{j=1}^nx_j^{\pi(j)-1},
    \prod_{j\in \sigma}x_j, x_{n+1}^{n - |\sigma|}\Big).
      \end{align*}
  The last term belongs to $\pi$ and $\sigma$, so that these two terms
  cancel each other in \eqref{eq:bigsum}.
  It is left to consider $(\pi,\sigma)$-terms for which there
  is no index $j \in \{1,\ldots,n\}$ in the sequence
  $\pi^{-1}(n),\pi^{-1}(n-1),\ldots,\pi^{-1}(1)$ such that $\pi^{-1}(j) \notin \sigma$ and $\pi^{-1}(j-1) \in \sigma$. For a given $\pi\in \mathcal{P}(n)$, the only sets
  $\sigma$ satisfying this condition are
  \begin{align*}
    \sigma_p=\{\pi^{-1}(n),\ldots,\pi^{-1}(p)\} \quad \text{for some} \;\;
    1\le p \le n+1,
  \end{align*}
  where we set $\sigma_{n+1}=\emptyset$. Therefore, we obtain from \eqref{eq:bigsum}
  \begin{equation} \label{eq:smallsum}
    \begin{aligned}
      & A\Big(\prod_{1\le j < i \le n+1}(x_i-x_j)\Big)\\
      &= \sum_{\pi \in \mathcal{P}_n}
    \sum_{p=1}^{n+1} \mathrm{sign}(\pi) (-1)^{n+1-p}
     A \Big( \prod_{j=1}^nx_j^{\pi(j)-1},\prod_{j \in \sigma_p} x_{j}, x_{n+1}^{p-1}\Big).
    \end{aligned}
  \end{equation}
  For $p \in \{1,\ldots,n+1\}$ and $\pi \in \mathcal{P}_n$ define
  $\tilde{\pi} \in \mathcal{P}_{n+1}$ by
  \begin{align*}
    \tilde{\pi}(j) = \begin{cases} \pi(j), & \pi(j) \in \{1,\ldots,p-1\},
      j \le n, \\
      \pi(j)+1, & \pi(j) \in \{p,\ldots,n\}, j\le n, \\
      p, & j=n+1.
    \end{cases}
  \end{align*}
  Clearly, the map $\mathcal{P}_n\times \{1,\ldots,n+1\} \to \mathcal{P}_{n+1}$,
  $(\pi,p) \mapsto \tilde{\pi}$ is a bijection which satisfies
  $\mathrm{sign}(\tilde{\pi})= (-1)^{n+p-1}\mathrm{sign}(\pi)$. With this setting
  we may write \eqref{eq:smallsum} as
  \begin{align*}
    A\Big(\prod_{1\le j < i \le n+1}(x_i-x_j)\Big)=
    \sum_{\tilde{\pi}\in \mathcal{P}_{n+1}} \mathrm{sign}(\tilde{\pi})
    A\Big( \prod_{j=1}^{n+1} x_j^{\tilde{\pi}(j)-1}\Big),
  \end{align*}
  which is our assertion.

  {\bf (ii):} \\
  Consider the right-hand side of \eqref{eq:sumV},use formula \eqref{eq:expandV}, and sort by the powers of $y$ to obtain
  \begin{align*}
     \sum_{i=1}^n& V(x_1,\ldots,x_{i-1},y,x_{i+1},\ldots,x_n) 
    = \sum_{i=1}^n \sum_{\pi \in \mathcal{P}_n} \mathrm{sign}(\pi)
    A \Big( \prod_{j=1, j\neq i}^n x_j^{\pi(j)-1}, y^{\pi(i)-1} \Big)\\
    & \quad \quad \qquad = \sum_{p=1}^n \sum_{\pi \in \mathcal{P}_n} \mathrm{sign}(\pi)
    A \Big( \prod_{j=1, \pi(j)\neq p}^n x_j^{\pi(j)-1}, y^{p-1} \Big).
  \end{align*}
  We consider two cases:\\
  {\bf $p=1$: } \\
    In this case we have
    \begin{align*}
      \sum_{\pi \in \mathcal{P}_n} \mathrm{sign}(\pi)
      A \Big( \prod_{j=1, \pi(j)\neq 1}^n x_j^{\pi(j)-1}, y^{p-1} \Big)&=
       \sum_{\pi \in \mathcal{P}_n} \mathrm{sign}(\pi)
      A \Big( \prod_{j=1}^n x_j^{\pi(j)-1} \Big)\\
      &= V(x_1,\ldots,x_n).
    \end{align*}

    \noindent
    {\bf $2 \le p \le n$:}\\
    We show that the corresponding terms vanish.
    Let $\ell=\pi^{-1}(p)$, $k =\pi^{-1}(1)$ (note $k \neq \ell$) and 
     define $\tilde{\pi}= \pi \circ \tau_{kl}$ as above, so
    that
    \begin{align*}
      \mathrm{sign}(\tilde{\pi})= - \mathrm{sign}(\pi), \quad
        \tilde{\pi}(\ell)=1, \quad \tilde{\pi}(k) = p.
    \end{align*}
    The corresponding $\tilde{\pi}$-term then satisfies
    \begin{align*}
      \mathrm{sign}(\tilde{\pi}) A\Big( \prod_{j=1, \tilde{\pi}(j) \neq p}^n
      x_j^{\tilde{\pi}(j)-1},  y^{p-1} \Big)= -  \mathrm{sign}&(\pi) A\Big(\prod_{j=1,j \neq k,\ell}^n
      x_j^{\pi(j)-1}, x_{\ell}^{\tilde{\pi}(\ell)-1}, y^{p-1} \Big)\\
    = - \mathrm{sign}(\pi)A\Big(\prod_{j=1,j \neq \ell}^n
      x_j^{\pi(j)-1},x_k^{\pi(k)-1}, y^{p-1} \Big) & = - \mathrm{sign}(\pi)A\Big(\prod_{j=1,j \neq \ell}^n
      x_j^{\pi(j)-1}, y^{p-1} \Big).
    \end{align*}
    Hence the terms which belong to $\pi$ and $\tilde{\pi}$ cancel each other, and
    the summand corresponding to $p \in \{ 2,\ldots,n\}$ vanishes.
\end{proof}

\subsection{Vandermonde (pseudo) $n$-metrics}
\label{sec3.2}
An easy consequence of Theorem \ref{th:gV} is the following construction of
a pseudo $n$-metric.
\begin{corollary} (The generalized Vandermonde pseudo $n$-metric) \label{cor:Vnmet}\\
  Let $X$ be a vector space and let $(Y,\|\cdot\|)$ be a normed space. Any $M_n$-linear
  and symmetric map $A:X^{M_n}\to Y$ defines a pseudo $n$-metric on $X$ via
  \begin{equation} \label{eq:defVmet}
    d(x_1,\ldots,x_n)= \big\| A\big( \prod_{1\le j <i\le n}(x_i-x_j) \big) \big\|,
    \quad  x_1,\ldots,x_n \in X.
  \end{equation}
  \end{corollary}
\begin{proof} The semidefiniteness and the nonnegativity are  obvious from the definition \eqref{eq:defVmet}. To prove symmetry we apply the representation \eqref{eq:expandV} for every $\sigma \in \mathcal{P}_n$:
  \begin{align*}
    d(x_{\sigma(1)},\ldots,x_{\sigma(n)})& = \big\| \sum_{\pi\in \mathcal{P}_n} \mathrm{sign}(\pi) A\Big( \prod_{j=1}^n
      x_{\sigma(j)}^{\pi(j)-1}\Big) \big\| \\
      & = \big\| \sum_{\pi\in \mathcal{P}_n} \mathrm{sign}(\pi) A\Big( \prod_{j=1}^n
      x_{j}^{(\pi \circ \sigma^{-1})(j)-1}\Big) \big\|\\
      & = \big\|\mathrm{sign}(\sigma) \sum_{\pi\in \mathcal{P}_n} \mathrm{sign}(\pi\circ \sigma^{-1}) A\Big( \prod_{j=1}^n
      x_{j}^{(\pi \circ \sigma^{-1})(j)-1}\Big) \big\|\\
      & = \big\| \sum_{\tilde{\pi}\in \mathcal{P}_n} \mathrm{sign}(\tilde{\pi}) A\Big( \prod_{j=1}^n
      x_{j}^{\tilde{\pi}(j)-1}\Big) \big\|\\
      & = d(x_1,\ldots,x_n).
  \end{align*}
  Finally, the simplex inequality follows by taking norms in \eqref{eq:sumV} and
  using the triangle inequality.
\end{proof}
The $n$-metrics from Theorem \ref{Cprop} are special cases of Corollary \ref{cor:Vnmet}
  for $\dim(X)=\dim(Y)=1,2$ when taking the multilinear form $A$ induced by real resp. complex multiplication.

There is even an algebraic identity which generalizes \eqref{eq:sumV} and leads to
an inequality which generalizes the extended simplex estimate \eqref{eq2:gsi}.

\begin{proposition} \label{prop7:extended}
  Let $X,Y$ be vector spaces over $\K=\R,\C$, let $q \in \{1,\ldots,n\}$ and let 
  \begin{align*}
    B: X^{M_n+q-1} \longrightarrow Y, \quad M_n = \frac{1}{2}n(n-1), n \ge 2
  \end{align*}
  be an $M_n+q-1$-linear and symmetric form. Then the map
  \begin{equation} \label{eq7:Bex}
    W(x_1,\ldots,x_n,x_{n+1}) = B\Big(\prod_{1\le j < i \le n}(x_i-x_j),x_{n+1}^{q-1}\Big)
  \end{equation}
  satisfies for all $x_1,\ldots,x_n,y \in X$ the following equality
  \begin{equation} \label{eq:sumW}
    W(x_1,\ldots,x_n,y)= \sum_{i=1}^n W(x_1,\ldots,x_{i-1},y,x_{i+1},\ldots,x_n,x_i).
  \end{equation}
  If $(Y,\|\cdot\|)$ is a normed space then we have for all $x_1,\ldots,x_n,y \in X$ the inequality
  \begin{equation} \label{eq7:simpW}
    \|W(x_1,\ldots,x_n,y)\|\le \sum_{i=1}^n \|W(x_1,\ldots,x_{i-1},y,x_{i+1},\ldots,x_n,x_i)\|,
    \end{equation}
\end{proposition}
\begin{remark} Note that \eqref{eq:sumW} and \eqref{eq7:simpW} simplify to \eqref{eq:sumV} and
  the simplex inequality in case $q=1$.
\end{remark}
\begin{proof} As in the proof of \eqref{eq:sumV} we begin with the right-hand sinde of \eqref{eq:sumW}.
  For any two indices $i,p \in \{1,\ldots,n\}$ we introduce partitionings of $\mathcal{P}_n$:
  \begin{align*}
    \mathcal{P}_n = \bigcup_{p=1}^n \mathcal{P}_{i,p}=\bigcup_{i=1}^n \mathcal{P}_{i,p}, \quad
    \mathcal{P}_{i,p}= \{\pi \in  \mathcal{P}_n: \pi(i)=p \}.
  \end{align*}
  Applying \eqref{eq:expandV} to $A=B(\cdot,x_i^{q-1})$ yields
  \begin{align*}
    \sum_{i=1}^n & W(x_1,\ldots,x_{i-1},y,x_{i+1},\ldots,x_n,x_i)\\
    &= \sum_{i=1}^n \sum_{\pi \in \mathcal{P}_n} \mathrm{sign}(\pi)
    B\Big(\prod_{j=1,j \neq i}^n x_j^{\pi(j)-1}, y^{\pi(i)-1}, x_i^{q-1}\Big)\\
    & = \sum_{i=1}^n \sum_{p=1}^n \sum_{\pi \in \mathcal{P}_{i,p}}\mathrm{sign}(\pi)
    B\Big(\prod_{j=1,j \neq i}^n x_j^{\pi(j)-1}, y^{\pi(i)-1}, x_i^{q-1}\Big)\\
    & = \sum_{i=1}^n \sum_{p=1}^n \sum_{\pi \in \mathcal{P}_{i,p}}\mathrm{sign}(\pi)
    B\Big(\prod_{j=1,\pi(j) \neq p}^n x_j^{\pi(j)-1}, x_{\pi^{-1}(p)}^{q-1}, y^{p-1}\Big)\\
    & =  \sum_{p=1}^n\left[ \sum_{i=1}^n\sum_{\pi \in \mathcal{P}_{i,p}}\mathrm{sign}(\pi)
      B\Big(\prod_{j=1,\pi(j) \neq p}^n x_j^{\pi(j)-1}, x_{\pi^{-1}(p)}^{q-1}, y^{p-1}\Big)\right] \\
    & =  \sum_{p=1}^n\left[\sum_{\pi \in \mathcal{P}_n}\mathrm{sign}(\pi)
      B\Big(\prod_{j=1,j \neq \pi^{-1}(p)}^n x_j^{\pi(j)-1}, x_{\pi^{-1}(p)}^{q-1}, y^{p-1}\Big)\right]
    =\sum_{p=1}^n S_p.
  \end{align*}
  We discuss the terms $S_p$ for $p=q$ and $p \neq q$.
  
  {\bf $p=q$:} With \eqref{eq7:Bex} and \eqref{eq:expandV} we obtain
  \begin{align*}
    S_q & = \sum_{\pi \in \mathcal{P}_n}\mathrm{sign}(\pi)
    B\Big(\prod_{j=1,j \neq \pi^{-1}(p)}^n x_j^{\pi(j)-1}, x_{\pi^{-1}(p)}^{\pi(\pi^{-1}(p))-1}, y^{q-1}\Big)\\
    & =  \sum_{\pi \in \mathcal{P}_n}\mathrm{sign}(\pi)
    B\Big(\prod_{j=1}^n x_j^{\pi(j)-1}, y^{q-1}\Big)= W(x_1,\ldots,x_n,y).
  \end{align*}

  {\bf $p \neq q$:} We are going to show $S_p=0$ so that our assertion follows from the case above.
  For a given $\pi \in \mathcal{P}_n$ we set $\ell=\pi^{-1}(p)$, $k=\pi^{-1}(q)$ and $\tilde{\pi}=\pi \circ \tau_{k \ell}$.
  By assumption we have $k \neq \ell$ and further
  \begin{align*}
    \mathrm{sign}(\tilde{\pi})= - \mathrm{sign}(\pi), \quad \tilde{\pi}(\ell)=\pi(k)=q, \quad \tilde{\pi}(k)= \pi(\ell)=p.
  \end{align*}
  Next we evaluate the summand in $S_p$ which belongs to $\tilde{\pi}$:
  \begin{align*}
   & \mathrm{sign}(\tilde{\pi})B \Big( \prod_{j=1,j \neq \tilde{\pi}^{-1}(p)}^n x_j^{\tilde{\pi}(j)-1},x_{\tilde{\pi}^{-1}(p)}^{q-1},y^{p-1}\Big)\\
    &=-\mathrm{sign}(\pi)B \Big( \prod_{j=1 , j \neq k}^n x_j^{\tilde{\pi}(j)-1},x_k^{\pi(k)-1},y^{p-1}\Big)\\
    & = -\mathrm{sign}(\pi)B \Big( \prod_{j=1 , j \neq \ell}^n x_j^{\pi(j)-1},x_{\ell}^{q-1},y^{p-1}\Big)\\
    &=  -\mathrm{sign}(\pi)B \Big( \prod_{j=1 , j \neq \pi^{-1}(p)}^n x_j^{\pi(j)-1},x_{\pi^{-1}(p)}^{q-1},y^{p-1}\Big).
  \end{align*}
  The latter is the negative of the summand which belongs to $\pi$. Hence, all these terms cancel in $S_p$ and
  the proof is complete.
\end{proof}
\subsection{ Definite $n$-metrics of Vandermonde type}
\label{sec3.3}
In this section we set up a  multilinear map which satisfies the assumptions of Corollary \ref{cor:Vnmet} and we analyze when definiteness holds.
\begin{example} \label{eq7:allprod}
  Let $C_k:(\R^2)^k \to \R^2$ denote the real $k$-linear form induced by complex multiplication,  i.e.
  \begin{equation*} \label{eq7:compmult}
    \begin{aligned}
      C_k(z_1,\ldots,z_k) &= J^{-1}\big(\prod_{j=1}^k J(z_j)\big), \quad \text{where} \\
      J\colon \R^2 \to \C, \;J(z_1,z_2) &= z_1 + i z_2, \quad z=(z_1,z_2) \in \R^2.
    \end{aligned}
  \end{equation*}
  Further consider the set of ordered pairs in $\{1,\ldots,m\}$
  \begin{align*}
    T_m = \{ \tau \in \{1,\ldots,m\}^2: \tau_1 < \tau_2 \}, \quad |T_m| = \frac{m(m-1)}{2}=M_m,
  \end{align*}
  and for $\tau \in T_m$ define the projection
  \begin{align*}
    P_{\tau}\colon \R^m \to \R^2, \quad P_{\tau}(x)= \begin{pmatrix} x_{\tau_1} \\ x_{\tau_2} \end{pmatrix}.
  \end{align*}
      With these preparations we define the $M_n$-linear map $A:(\R^m)^{M_n} \to \R^{2 M_m}$ by
  \begin{equation*} \label{eq7:Adef}
    A(x_1,\ldots,x_{M_n})_{\tau}= C_{M_n}(P_{\tau}x_1,\ldots,P_{\tau}x_{M_n}), \quad \tau \in T_m.
  \end{equation*}
 Since $C_{M_n}$ is symmetric, the same follows for the map $A$ :
  \begin{align*}
    A(x_{\pi(1)},\ldots,x_{\pi(M_n)})_{\tau} &= C_{M_n}((P_{\tau}x_{\pi(1)},\ldots,P_{\tau}x_{\pi(M_n)}) \\
    & =C_{M_n}(P_{\tau}x_{1},\ldots,P_{\tau}x_{M_n})=A(x_1,\ldots,x_{M_n})_{\tau}.
  \end{align*}
  Hence Corollary \ref{cor:Vnmet} applies and  any norm $\|\cdot\|$ in $\R^{2 M_m}$ leads to a pseudo $n$-metric
  in $\R^m$ via
  \begin{equation} \label{eq7:defnmpseudo}
    d(x_1,\ldots,x_n)= \big\| A\big( \prod_{1\le j <i\le n}(x_i-x_j) \big) \big\|,
  \end{equation}
  \begin{proposition} \label{s3:prop2}
 The equation \eqref{eq7:defnmpseudo}  defines a $3$-metric in case  $m \ge n=3$. 
  \end{proposition}
  \begin{proof} Assume $d(x_1,x_2,x_3)=0$, i.e.
  $A(x_2-x_1, x_3-x_2,x_3-x_1)_{\tau}=0$ for all $\tau \in T_m$. By the definiteness of complex multiplication we obtain
  that there exists a map $j:T_m \to T_3$, $\tau \mapsto j(\tau)=(j_1,j_2)(\tau)$ such that $ P_{\tau}(x_{j_2(\tau)}-x_{j_1(\tau)})=0$
  for all $\tau \in T_m$
   Define the index sets
  \begin{align*}
    \sigma_{12}&=\{\tau_1,\tau_2: j(\tau)=(1,2), \tau \in T_m\},
    \quad \sigma_{13}=\{\tau_1,\tau_2: j(\tau)=(1,3), \tau \in T_m\},\\
    \quad \sigma_{23}&=\{\tau_1,\tau_2: j(\tau)=(2,3), \tau \in T_m\}.
  \end{align*}
  By assumption we have $\sigma_{12}\cup \sigma_{13} \cup \sigma_{23} = \{1,\ldots,m\}$. If one of these sets
  equals $\{1,\ldots,m\}$ then one of the differences $x_2-x_1$, $x_3-x_1$, $x_3-x_2$ vanishes and our assertion
  follows. Hence we can assume
  $\sigma_{12},\sigma_{13}, \sigma_{23} \neq \{1,\ldots,m\}$.  W.l.o.g. let $\sigma_{12}$ have  the maximum
  cardinality, i.e. $|\sigma_{12}| \ge |\sigma_{13}|,|\sigma_{23}|$. Now take any $\tau_1\notin \sigma_{12}$
  and note that such an index exists.
  Then $\tau_1$ must lie in $\sigma_{13}$ or $\sigma_{23}$. If $\tau_1 \in \sigma_{13}$ we claim there exists
  $\tau_2\in \sigma_{12}$ with $\tau_2 \notin \sigma_{13}$, for otherwise $\sigma_{12}\subseteq \sigma_{13}$
  and $|\sigma_{13}|\ge |\sigma_{12}|+1$ (since $\tau_1 \notin \sigma_{12}$, $\tau_1 \in \sigma_{13}$) leads to a contradiction.
  Similarly, if $\tau_1 \in \sigma_{23}$ there exists $\tau_2 \in \sigma_{12}$ with $\tau_2 \notin \sigma_{23}$.
  In the following assume w.l.o.g. that $\tau_1\in \sigma_{13}$ holds. Since the pair $\tau_1,\tau_2$ lies neither
  in $\sigma_{12}$ nor in $\sigma_{13}$ we have $\tau_1,\tau_2 \in \sigma_{23}$. In particular, we have the following
  zero components
  \begin{align*}
    0=(x_2-x_1)_{\tau_2}= (x_3-x_1)_{\tau_1}= (x_3-x_2)_{\tau_1}=(x_3-x_2)_{\tau_2}.
  \end{align*}
  This implies $(x_2-x_1)_{\tau_1}= (x_3-x_1)_{\tau_1}- (x_3-x_2)_{\tau_1}=0$ for all $\tau_1 \notin
  \sigma_{12}$. 
    Since $(x_2-x_1)_{\tau}=0$ holds for all $\tau \in \sigma_{12}$ we have shown $x_2-x_1=0$.
  \end{proof}
  The proof shows that the  question of definiteness for general values of $n,m$ leads to a combinatorial problem. We have no general result for this problem, but note that
in case $n=m=4$, definiteness does not hold as the following counterexample shows:
  \begin{align*}
    \begin{pmatrix} x_1 & x_2 & x_3 & x_4 \end{pmatrix}&=
    \begin{pmatrix} 0 & 0 & 0 & 1 \\ 0 & 0 & 1 & 0 \\ 0 & 0 & 1 & 1 \\ 0& -1 & 0 & 0 \end{pmatrix},
  \end{align*}
  \begin{align*}
   ( x_2-x_1 \quad x_3-x_1 \quad x_3 -x_2 \quad & x_4-x_1 \quad x_4-x_2 \quad x_4-x_3 )\\
   & = \begin{pmatrix} 0 & 0 & 0 &1 & 1 & 1 \\ 0 & 1 & 1 & 0 & 0 & -1 \\
      0 & 1 & 1 & 1 & 1 &0 \\- 1 & 0 &1 & 0 & 1& 0 \end{pmatrix}.
  \end{align*}
  For every pair $1 \le j < i \le 4$ there exists a column with $i$-th and $j$-th entry zero.
  Therefore $d(x_1,x_2,x_3,x_4)=0$ holds but the vectors $x_1,x_2,x_3,x_4$ are pairwise distinct.
\end{example}
\subsection{An application to linear differential equations}
\label{sec3.4}
We conclude with a simple estimate of the $3$-metric \eqref{s1:e4}
for the solutions of a linear ODE system
\begin{equation} \label{eq3:ODE}
  \dot{x}(t) = A(t) x(t), \quad t \ge 0.
\end{equation}
Let  $A:[0,\infty) \to \R^{m,m}$ be a continuous matrix function satisfying
  \begin{equation*} \label{eq3:Aest}
    \langle A(t)x,x\rangle \le \alpha(t) \langle x, x \rangle \quad \forall t\ge 0, x \in \R^m
  \end{equation*}
  for some continuous function $\alpha:[0,\infty) \to \R$.
  Let  $x_j(t)$, $t\ge 0$ be the solutions of \eqref{eq3:ODE} with initial
  data $x_j(0)=x_j^0$, $j=1,2,3$.
  Then the following estimate holds for the $3$-metric
  \begin{equation} \label{eq3:solest}
    d_3(x_1(t),x_2(t),x_3(t)) \le \exp \left( 3 \int_0^t \alpha(s) ds\right) d_3(x_1^0,x_2^0,x_3^0), \quad t \ge 0.
  \end{equation}
  In case $d_3(x_1^0,x_2^0,x_3^0)=0$ the estimate \eqref{eq3:solest} is obvious since then
  $d_3(x_1(t),x_2(t),x_3(t))=0$ holds for all $t \ge 0$ by the definiteness of $d_3$ and the unique solvability
  of the initial value problem. If $d_3(x_1^0,x_2^0,x_3^0)>0$ we obtain by differentiation
  \begin{align*}
    \frac{d}{dt}d_3(x_1 ,x_2 ,x_3 )& = \sum_{1\le \ell < k\le 3}\frac{d}{dt}\|x_k -x_{\ell} \|
    \prod_{ 1 \le j < i \le 3 \atop  (k,\ell)\neq (i,j) } \|x_i -x_j \| \\
    & =\sum_{1\le \ell < k\le 3}\frac{\langle A(t) (x_k -x_{\ell} ),x_k -x_{\ell}  \rangle}{\|x_k -x_{\ell} \|}
    \prod_{ 1 \le j < i \le 3 \atop  (k,\ell)\neq (i,j) } \|x_i -x_j \| \\
    & \le \sum_{1\le \ell < k\le 3} \alpha(t) \|x_k -x_{\ell}\| 
    \prod_{ 1 \le j < i \le 3 \atop  (k,\ell)\neq (i,j) } \|x_i -x_j \| = 3 \alpha(t) d_3(x_1,x_2,x_3).
  \end{align*}
 Inequality \eqref{eq3:solest} then follows by a Gronwall estimate.

\section*{Acknowledgments}
I thank Th. H\"uls and J.-L. Marichal for pointing out some references to me and,
 furthermore, Th. H\"uls for support with the figures.
\bigskip


\begin{thebibliography}{14}
\ifx \bisbn   \undefined \def \bisbn  #1{ISBN #1}\fi
\ifx \binits  \undefined \def \binits#1{#1}\fi
\ifx \bauthor  \undefined \def \bauthor#1{#1}\fi
\ifx \batitle  \undefined \def \batitle#1{#1}\fi
\ifx \bjtitle  \undefined \def \bjtitle#1{#1}\fi
\ifx \bvolume  \undefined \def \bvolume#1{\textbf{#1}}\fi
\ifx \byear  \undefined \def \byear#1{#1}\fi
\ifx \bissue  \undefined \def \bissue#1{#1}\fi
\ifx \bfpage  \undefined \def \bfpage#1{#1}\fi
\ifx \blpage  \undefined \def \blpage #1{#1}\fi
\ifx \burl  \undefined \def \burl#1{\textsf{#1}}\fi
\ifx \doiurl  \undefined \def \doiurl#1{\url{https://doi.org/#1}}\fi
\ifx \betal  \undefined \def \betal{\textit{et al.}}\fi
\ifx \binstitute  \undefined \def \binstitute#1{#1}\fi
\ifx \binstitutionaled  \undefined \def \binstitutionaled#1{#1}\fi
\ifx \bctitle  \undefined \def \bctitle#1{#1}\fi
\ifx \beditor  \undefined \def \beditor#1{#1}\fi
\ifx \bpublisher  \undefined \def \bpublisher#1{#1}\fi
\ifx \bbtitle  \undefined \def \bbtitle#1{#1}\fi
\ifx \bedition  \undefined \def \bedition#1{#1}\fi
\ifx \bseriesno  \undefined \def \bseriesno#1{#1}\fi
\ifx \blocation  \undefined \def \blocation#1{#1}\fi
\ifx \bsertitle  \undefined \def \bsertitle#1{#1}\fi
\ifx \bsnm \undefined \def \bsnm#1{#1}\fi
\ifx \bsuffix \undefined \def \bsuffix#1{#1}\fi
\ifx \bparticle \undefined \def \bparticle#1{#1}\fi
\ifx \barticle \undefined \def \barticle#1{#1}\fi
\bibcommenthead
\ifx \bconfdate \undefined \def \bconfdate #1{#1}\fi
\ifx \botherref \undefined \def \botherref #1{#1}\fi
\ifx \url \undefined \def \url#1{\textsf{#1}}\fi
\ifx \bchapter \undefined \def \bchapter#1{#1}\fi
\ifx \bbook \undefined \def \bbook#1{#1}\fi
\ifx \bcomment \undefined \def \bcomment#1{#1}\fi
\ifx \oauthor \undefined \def \oauthor#1{#1}\fi
\ifx \citeauthoryear \undefined \def \citeauthoryear#1{#1}\fi
\ifx \endbibitem  \undefined \def \endbibitem {}\fi
\ifx \bconflocation  \undefined \def \bconflocation#1{#1}\fi
\ifx \arxivurl  \undefined \def \arxivurl#1{\textsf{#1}}\fi
\csname PreBibitemsHook\endcsname

\bibitem[\protect\citeauthoryear{G\"ahler}{1963}]{Gaehler63}
\begin{barticle}
\bauthor{\bsnm{G\"ahler}, \binits{S.}}:
\batitle{{$2$}-metrische {R}\"aume und ihre topologische {S}truktur}.
\bjtitle{Math. Nachr.}
\bvolume{26},
\bfpage{115}--\blpage{148}
(\byear{1963})
\doiurl{10.1002/mana.19630260109}
\end{barticle}
\endbibitem

\bibitem[\protect\citeauthoryear{G\"ahler}{1964}]{Gaehler64}
\begin{barticle}
\bauthor{\bsnm{G\"ahler}, \binits{S.}}:
\batitle{Lineare {$2$}-normierte {R}\"aume}.
\bjtitle{Math. Nachr.}
\bvolume{28},
\bfpage{1}--\blpage{43}
(\byear{1964})
\doiurl{10.1002/mana.19640280102}
\end{barticle}
\endbibitem

\bibitem[\protect\citeauthoryear{G\"ahler}{1969}]{Gaehler69}
\begin{barticle}
\bauthor{\bsnm{G\"ahler}, \binits{S.}}:
\batitle{Untersuchungen \"uber verallgemeinerte {$m$}-metrische {R}\"aume, {I},
  {II}, {III}}.
\bjtitle{Math. Nachr.}
\bvolume{40},
\bfpage{165}--\blpage{18940229264412336}
(\byear{1969})
\doiurl{10.1002/mana.19690410103}
\end{barticle}
\endbibitem

\bibitem[\protect\citeauthoryear{Kiss et~al.}{2018}]{Marichal18}
\begin{barticle}
\bauthor{\bsnm{Kiss}, \binits{G.}},
\bauthor{\bsnm{Marichal}, \binits{J.-L.}},
\bauthor{\bsnm{Teheux}, \binits{B.}}:
\batitle{A generalization of the concept of distance based on the simplex
  inequality}.
\bjtitle{Beitr. Algebra Geom.}
\bvolume{59}(\bissue{2}),
\bfpage{247}--\blpage{266}
(\byear{2018})
\doiurl{10.1007/s13366-018-0379-5}
\end{barticle}
\endbibitem

\bibitem[\protect\citeauthoryear{Kiss and Marichal}{2023}]{Marichal23}
\begin{barticle}
\bauthor{\bsnm{Kiss}, \binits{G.}},
\bauthor{\bsnm{Marichal}, \binits{J.-L.}}:
\batitle{Nonstandard {$n$}-distances based on certain geometric constructions}.
\bjtitle{Beitr. Algebra Geom.}
\bvolume{64}(\bissue{1}),
\bfpage{107}--\blpage{126}
(\byear{2023})
\doiurl{10.1007/s13366-022-00623-5}
\end{barticle}
\endbibitem

\bibitem[\protect\citeauthoryear{Menger}{1928}]{Menger28}
\begin{barticle}
\bauthor{\bsnm{Menger}, \binits{K.}}:
\batitle{Untersuchungen \"uber allgemeine {M}etrik}.
\bjtitle{Math. Ann.}
\bvolume{100}(\bissue{1}),
\bfpage{75}--\blpage{163}
(\byear{1928})
\doiurl{10.1007/BF01448840}
\end{barticle}
\endbibitem

\bibitem[\protect\citeauthoryear{Huang and Tan}{2018}]{huangtan18}
\begin{barticle}
\bauthor{\bsnm{Huang}, \binits{X.}},
\bauthor{\bsnm{Tan}, \binits{D.}}:
\batitle{Mappings of preserving {$n$}-distance one in {$n$}-normed spaces}.
\bjtitle{Aequationes Math.}
\bvolume{92}(\bissue{3}),
\bfpage{401}--\blpage{413}
(\byear{2018})
\doiurl{10.1007/s00010-018-0539-6}
\end{barticle}
\endbibitem

\bibitem[\protect\citeauthoryear{Huang and Tan}{2019}]{huangtan19}
\begin{barticle}
\bauthor{\bsnm{Huang}, \binits{X.}},
\bauthor{\bsnm{Tan}, \binits{D.}}:
\batitle{A {T}ingley's type problem in {$n$}-normed spaces}.
\bjtitle{Aequationes Math.}
\bvolume{93}(\bissue{5}),
\bfpage{905}--\blpage{918}
(\byear{2019})
\doiurl{10.1007/s00010-019-00637-w}
\end{barticle}
\endbibitem

\bibitem[\protect\citeauthoryear{Van~An et~al.}{2015}]{VVKR15}
\begin{barticle}
\bauthor{\bsnm{Van~An}, \binits{T.}},
\bauthor{\bsnm{Van~Dung}, \binits{N.}},
\bauthor{\bsnm{Kadelburg}, \binits{Z.}},
\bauthor{\bsnm{Radenovi\'c}, \binits{S.}}:
\batitle{Various generalizations of metric spaces and fixed point theorems}.
\bjtitle{Rev. R. Acad. Cienc. Exactas F\'is. Nat. Ser. A Mat. RACSAM}
\bvolume{109}(\bissue{1}),
\bfpage{175}--\blpage{198}
(\byear{2015})
\doiurl{10.1007/s13398-014-0173-7}
\end{barticle}
\endbibitem

\bibitem[\protect\citeauthoryear{Schwaiger}{2023}]{Schwaiger23}
\begin{barticle}
\bauthor{\bsnm{Schwaiger}, \binits{J.}}:
\batitle{On the existence of {$m$}-norms in vector spaces over valued fields}.
\bjtitle{Aequationes Math.}
\bvolume{97}(\bissue{5-6}),
\bfpage{1051}--\blpage{1058}
(\byear{2023})
\doiurl{10.1007/s00010-023-00956-z}
\end{barticle}
\endbibitem

\bibitem[\protect\citeauthoryear{Kiss and Marichal}{2020}]{Marichal20}
\begin{barticle}
\bauthor{\bsnm{Kiss}, \binits{G.}},
\bauthor{\bsnm{Marichal}, \binits{J.-L.}}:
\batitle{On the best constants associated with {$n$}-distances}.
\bjtitle{Acta Math. Hungar.}
\bvolume{161}(\bissue{1}),
\bfpage{341}--\blpage{365}
(\byear{2020})
\doiurl{10.1007/s10474-020-01023-8}
\end{barticle}
\endbibitem

\bibitem[\protect\citeauthoryear{Beyn}{2024}]{Be24}
\begin{barticle}
\bauthor{\bsnm{Beyn}, \binits{W.-J.}}:
\batitle{On a generalized notion of metrics}.
\bjtitle{Aequationes Math.}
\bvolume{98}(\bissue{4}),
\bfpage{953}--\blpage{977}
(\byear{2024}).
\bcomment{arXiv:2305.17081v2}
\end{barticle}
\endbibitem

\bibitem[\protect\citeauthoryear{Svrtan and Veljan}{2012}]{SV12}
\begin{barticle}
\bauthor{\bsnm{Svrtan}, \binits{D.}},
\bauthor{\bsnm{Veljan}, \binits{D.}}:
\batitle{Non-{E}uclidean versions of some classical triangle inequalities}.
\bjtitle{Forum Geom.}
\bvolume{12},
\bfpage{197}--\blpage{209}
(\byear{2012})
\end{barticle}
\endbibitem

\bibitem[\protect\citeauthoryear{Horn and Johnson}{2013}]{HJ2013}
\begin{bbook}
\bauthor{\bsnm{Horn}, \binits{R.A.}},
\bauthor{\bsnm{Johnson}, \binits{C.R.}}:
\bbtitle{Matrix Analysis},
\bedition{2}nd edn.,
p. \bfpage{643}.
\bpublisher{Cambridge University Press, Cambridge}, 
(\byear{2013})
\end{bbook}
\endbibitem

\end{thebibliography}
\end{document}